\documentclass[12pt]{amsart}
\usepackage{amscd,amssymb,latexsym}
\usepackage{fullpage}
\newcommand{\Mdef}[2]{\newcommand{#1}{\relax \ifmmode #2 \else $#2$\fi}}

\newcommand{\rank}{\mathrm{rank}}

\newcommand{\tensor}{\otimes}

\Mdef{\bhom}{\mathbf{\hat{H}om}}

\Mdef{\Mod}{\mathrm{mod}}

\newcommand{\st}{\; | \;}

\newtheorem{thm}{Theorem}[section]
\newtheorem{lemma}[thm]{Lemma}
\newtheorem{prop}[thm]{Proposition}
\newtheorem{cor}[thm]{Corollary}

\theoremstyle{definition}

\newcommand{\qqed}{\qed \\[1ex]}
\renewenvironment{proof}[1][\hspace*{-.8ex}]{\noindent {\bf Proof #1:\;}}{\qqed}

\Mdef{\PH} {\Phi^H}
\Mdef{\PK} {\Phi^K}
\Mdef{\PL} {\Phi^L}
\Mdef{\PT} {\Phi^{\T}}

\Mdef{\ef}{E{\cF}_+}
\Mdef{\etf}{\widetilde{E}{\cF}}
\Mdef{\eg}{E{G}_+}
\Mdef{\etg}{\tilde{E}{G}}

\Mdef{\infl}{\mathrm{inf}}
\Mdef{\defl}{\mathrm{def}}
\Mdef{\res}{\mathrm{res}}
\Mdef{\ind}{\mathrm{ind}}
\Mdef{\coind}{\mathrm{coind}}

\Mdef{\univ}{\mathcal{U}}

\Mdef{\Fp}{\mathbb{F}_p}
\Mdef{\Zpinfty}{\Z /p^{\infty}}
\Mdef{\Zpadic}{\Z_p^{\wedge}}

\newcommand{\lra}{\longrightarrow}

\newcommand{\lr}[1]{\langle #1 \rangle}

\newcommand{\Gspectra}{\mbox{$G$-{\bf spectra}}}

\newcommand{\spec}{\mathrm{Spec}}

\Mdef{\we}{\mathbf{we}}
\Mdef{\fib}{\mathbf{fib}}
\Mdef{\cof}{\mathbf{cof}}
\Mdef{\BI}{\mathcal{BI}}

\Mdef{\B}{\mathbb{B}}
\Mdef{\C}{\mathbb{C}}
\Mdef{\D}{\mathbb{D}}
\Mdef{\E}{\mathbb{E}}
\Mdef{\T}{\mathbb{T}}
\Mdef{\F}{\mathbb{F}}
\Mdef{\G}{\mathbb{G}}
\Mdef{\I}{\mathbb{I}}
\Mdef{\N}{\mathbb{N}}
\Mdef{\Q}{\mathbb{Q}}
\Mdef{\R}{\mathbb{R}}
\Mdef{\bbS}{\mathbb{S}}
\Mdef{\Z}{\mathbb{Z}}

\Mdef{\bA}{\mathbb{A}}
\Mdef{\bB}{\mathbb{B}}
\Mdef{\bC}{\mathbb{C}}
\Mdef{\bD}{\mathbb{D}}
\Mdef{\bE}{\mathbb{E}}
\Mdef{\bF}{\mathbb{F}}
\Mdef{\bG}{\mathbb{G}}
\Mdef{\bH}{\mathbb{H}}
\Mdef{\bI}{\mathbb{I}}
\Mdef{\bJ}{\mathbb{J}}
\Mdef{\bK}{\mathbb{K}}
\Mdef{\bL}{\mathbb{L}}
\Mdef{\bM}{\mathbb{M}}
\Mdef{\bN}{\mathbb{N}}
\Mdef{\bO}{\mathbb{O}}
\Mdef{\bP}{\mathbb{P}}
\Mdef{\bQ}{\mathbb{Q}}
\Mdef{\bR}{\mathbb{R}}
\Mdef{\bS}{\mathbb{S}}
\Mdef{\bT}{\mathbb{T}}
\Mdef{\bU}{\mathbb{U}}
\Mdef{\bV}{\mathbb{V}}
\Mdef{\bW}{\mathbb{W}}
\Mdef{\bX}{\mathbb{X}}
\Mdef{\bY}{\mathbb{Y}}
\Mdef{\bZ}{\mathbb{Z}}

\Mdef{\cA}{\mathcal{A}}
\Mdef{\cB}{\mathcal{B}}
\Mdef{\cC}{\mathcal{C}}
\Mdef{\mcD}{\mathcal{D}} 
\Mdef{\cE}{\mathcal{E}}
\Mdef{\cF}{\mathcal{F}}
\Mdef{\cG}{\mathcal{G}}
\Mdef{\mcH}{\mathcal{H}} 
\Mdef{\cI}{\mathcal{I}}
\Mdef{\cJ}{\mathcal{J}}
\Mdef{\cK}{\mathcal{K}}
\Mdef{\mcL}{\mathcal{L}}

\Mdef{\cM}{\mathcal{M}}
\Mdef{\cN}{\mathcal{N}}
\Mdef{\cO}{\mathcal{O}}
\Mdef{\cP}{\mathcal{P}}
\Mdef{\cQ}{\mathcal{Q}}
\Mdef{\mcR}{\mathcal{R}}
\Mdef{\cS}{\mathcal{S}}
\Mdef{\cT}{\mathcal{T}}
\Mdef{\cU}{\mathcal{U}}
\Mdef{\cV}{\mathcal{V}}
\Mdef{\cW}{\mathcal{W}}
\Mdef{\cX}{\mathcal{X}}
\Mdef{\cY}{\mathcal{Y}}
\Mdef{\cZ}{\mathcal{Z}}

\Mdef{\ca}{\mathcal{a}}
\Mdef{\ct}{\mathcal{t}}

\Mdef{\At}{\widetilde{A}}
\Mdef{\Bt}{\tilde{B}}
\Mdef{\Ct}{\tilde{C}}
\Mdef{\Dt}{\widetilde{D}}
\Mdef{\Et}{\tilde{E}}
\Mdef{\Ht}{\tilde{H}}
\Mdef{\Kt}{\tilde{K}}
\Mdef{\Lt}{\tilde{L}}
\Mdef{\Mt}{\tilde{M}}
\Mdef{\Nt}{\tilde{N}}
\Mdef{\Pt}{\tilde{P}}
\newcommand{\Sigmat}{\widetilde{\Sigma}}

\Mdef{\tA}{\tilde{A}}
\Mdef{\tB}{\tilde{B}}
\Mdef{\tC}{\tilde{C}}
\Mdef{\tE}{\tilde{E}}
\Mdef{\tH}{\tilde{H}}
\Mdef{\tK}{\tilde{K}}
\Mdef{\tL}{\tilde{L}}
\Mdef{\tM}{\tilde{M}}
\Mdef{\tN}{\tilde{N}}
\Mdef{\tP}{\tilde{P}}

\Mdef{\ft}{\tilde{f}}
\Mdef{\xt}{\tilde{x}}
\Mdef{\yt}{\tilde{y}}

\Mdef{\Ab}{\overline{A}}
\Mdef{\Bb}{\overline{B}}
\Mdef{\Cb}{\overline{C}}
\Mdef{\Db}{\overline{D}}
\Mdef{\Eb}{\overline{E}}
\Mdef{\Fb}{\overline{F}}
\Mdef{\Gb}{\overline{G}}
\Mdef{\Hb}{\overline{H}}
\Mdef{\Ib}{\overline{I}}
\Mdef{\Jb}{\overline{J}}
\Mdef{\Kb}{\overline{K}}
\Mdef{\Lb}{\overline{L}}
\Mdef{\Mb}{\overline{M}}
\Mdef{\Nb}{\overline{N}}
\Mdef{\Ob}{\overline{O}}
\Mdef{\Pb}{\overline{P}}
\Mdef{\Qb}{\overline{Q}}
\Mdef{\Rb}{\overline{R}}
\Mdef{\Sb}{\overline{S}}
\Mdef{\Tb}{\overline{T}}
\Mdef{\Ub}{\overline{U}}
\Mdef{\Vb}{\overline{V}}
\Mdef{\Wb}{\overline{W}}
\Mdef{\Xb}{\overline{X}}
\Mdef{\Yb}{\overline{Y}}
\Mdef{\Zb}{\overline{Z}}

\Mdef{\db}{\overline{d}}
\Mdef{\hb}{\overline{h}}
\Mdef{\qb}{\overline{q}}
\Mdef{\rb}{\overline{r}}
\Mdef{\tb}{\overline{t}}
\Mdef{\ub}{\overline{u}}
\Mdef{\vb}{\overline{v}}

\Mdef{\hc}{\hat{c}}
\Mdef{\he}{\hat{e}}
\Mdef{\hf}{\hat{f}}
\Mdef{\hA}{\hat{A}}
\Mdef{\hH}{\hat{H}}
\Mdef{\hJ}{\hat{J}}
\Mdef{\hM}{\hat{M}}
\Mdef{\hP}{\hat{P}}
\Mdef{\hQ}{\hat{Q}}

\Mdef{\thetab}{\overline{\theta}}
\Mdef{\phib}{\overline{\phi}}

\Mdef{\uA}{\underline{A}}
\Mdef{\uB}{\underline{B}}
\Mdef{\uC}{\underline{C}}
\Mdef{\uD}{\underline{D}}

\Mdef{\bolda}{\mathbf{a}}
\Mdef{\boldb}{\mathbf{b}}
\Mdef{\bfD}{\mathbf{D}}

\Mdef{\fm}{\frak{m}}
\Mdef{\fp}{\frak{p}}

\newcommand{\fX}{\mathfrak{X}}

\Mdef{\eps}{\epsilon}

\input{xypic}
\newcommand{\sub}{\mathrm{Sub}}

\renewcommand{\tb}{\overline{\times}}

\newcommand{\diag}{\mathrm{diag}}

\newcommand{\Qt}{\tilde{\Q}}
\newcommand{\toral}{\mathrm{toral}}
\newcommand{\gqone}{\mathrm{gq1}}
\newcommand{\gqwf}{\mathrm{gqwf}}
\newcommand{\discrete}{\mathrm{Discrete}}
\newcommand{\mix}{\mathrm{mixed}}
\newcommand{\gqtoral}{\mathrm{gqtoral}}
\newcommand{\aut}{\mathrm{Aut}}
\newcommand{\out}{\mathrm{Out}}
\newcommand{\Hhat}{\hat{H}}

\begin{document}
\title{Rational $Sp(2)$-equivariant cohomology theories I: dominant subgroups}

\author{J.P.C.Greenlees}
\address{Mathematics Institute, Zeeman Building, Coventry CV4, 7AL, UK}
\email{john.greenlees@warwick.ac.uk}

\date{}

\begin{abstract}
We give a general description of the spectral space of conjugacy classes of subgroups of
$Sp(2)$: it is a disjoint union of finitely many 
blocks, each dominated by a subgroup: of these blocks, 26 are of dimension 1,
6 are of dimension 2 and the remainder are isolated points. On each
of these blocks there is a sheaf of polynomial rings
and a component structure. These are the ingredients for constructing an abelian
category $\cA (Sp(2))$ designed to reflect the structure of rational
$Sp(2)$-equivariant cohomology theories. We assemble the
results from earlier papers in the series to show that the category of rational
$Sp(2)$-spectra  is Quillen equivalent to the category of differential
graded objects of $\cA (Sp(2))$. In the sequel \cite{sp2q2} we will make the fine
structure of $\cA (Sp(2))$ explicit, and make calculations based upon it.

\end{abstract}
\thanks{
 The work is partially supported by EPSRC Grant
  EP/W036320/1.}
\maketitle

\tableofcontents

\section{Introduction}
\subsection{Context}
It is conjectured \cite{AGconj} that for each compact Lie group $G$ there is an abelian
category $\cA (G)$ and a Quillen equivalence between the category of
rational $G$-spectra and the category of differential graded objects
of $\cA(G)$: 
$$\Gspectra \simeq_Q \mbox{DG-}\cA (G). $$
This is known for a range of small groups including all subgroups of
$SU(3)$:  the purpose of this note is to prove the conjecture for
$G=Sp(2)$ and all its subgroups. This shows that calculations for
$G$-spectra can be made using $\cA (G)$: in the sequel 
\cite{sp2q2} we will make the fine structure of $\cA (Sp(2))$
explicit,  and apply it  to make explicit 
calculations. 

This paper uses methods that apply quite
generally. Implementation of the methods requires doing group theory
and representation theory for the individual groups. Because of this, 
we focus on  $Sp(2)$ but there is plenty of
information for other rank 2 groups. Indeed, we recover the case of
$SU(3)$ from \cite{su3q} by these more uniform methods and we give complete
information for  $SU(2)\times SU(2)$. The case of $G_2$ is complicated
enough to merit separate treatment elsewhere.

\subsection{Strategy}
In effect the algebraic model $\cA (G)$ is built by assembling data for each conjugacy
class of subgroups $H\subseteq G$. Indeed, the data over $(H)$ is a
model for $G$-spectra with geometric isotropy concentrated on the
single conjugacy class $(H)$, which is equivalent to  free
$W_G(H)$-spectra \cite{spcgq}  and this has the  model $\cA (G|H)$
consisting of torsion modules over the
twisted group ring $H^*(BW_G^e(H))[W_G^d(H)]$ \cite{gfreeq2}, where
the Weyl group $W_G(H)=N_G(H)/H$ has identity component $W_G^e(H)$ and
discrete quotient $W_G^d(H)=\pi_0(W_G(H))$. We will
view this as stating that $\cA (G)$ consists of equivariant `sheaves' over the space
$\fX_G=\sub(G)/G$ of conjugacy classes of subgroups of $G$, but the precise
meaning of the word `sheaves' and the additional structure on these
`sheaves' needs some elucidation. 

In any case, this form suggests that if $H$ is a subgroup of $G$ and  we restrict sheaves over
$\sub(G)/G$ to conjugacy classes of subgroups of  $H$, we may
expect the category to be closely related to $\cA (H)$, though of
course the fusion of conjugacy classes along the map $\fX_H=\sub(H)/H\lra
\sub(G)/G=\fX_G$ will need to be taken into account, along with the
transition from $W_H(K)$ to the larger group  $W_G(K)$. This in turn means that when
constructing $\cA (G)$ it is natural to adopt an inductive approach
and begin by giving a construction of $\cA (H)$ for all subgroups $H$
of $G$. We have already dealt with most proper subgroups  $H$ of
$Sp(2)$ in their own right,
and this paper adapts that work to understand them as subgroups of $Sp(2)$.

\subsection{Partitions}  The fact that  $\sub (G)/G$ is the
Balmer spectrum of finite rational $G$-spectra \cite{spcgq} suggests some of the
relevant additional structure. The Balmer spectrum is equipped
with the Zariski topology, and we write $\fX_G=\sub (G)/G$ for this
space. As described in \cite{prismatic} we may use the language of
Priestley spaces, and state that $\fX_G$ has underlying constructible
topology on $\sub (G)/G$ being the h-topology (the quotient topology of the
Hausdorff metric topology on $\sub(G)$), and the spectral ordering is
the cotoral ordering\footnote{$K$ is {\em cotoral} in $H$ if $K$ is normal
  in $H$ with quotient a torus. }. Thus the closed sets of $\fX_G$ are precisely
the  $h$-closed sets
closed under cotoral specialization.

One may show in general that $\fX_G$ admits a partition into Zariski clopen blocks of
rather standard forms: thus
 $$\fX_G=\cV_1^G \amalg  \cV_2^G\amalg 
\cdots \amalg \cV_n^G,$$
 where each summand $\cV_i^G$ is `dominated' by a 
subgroup $H=H_i$ of $G$.  For each of these subroups $H$, the Weyl group $W_G(H)$ is
finite. The block $\cV^G_H$ is constructed as follows: we choose an h-neighbourhood $\cN^H_H$ of $H$
in $\Phi (H)$ (conjugacy classes of subgroups of with finite Weyl
groups).  We choose $\cN^H_H$ small enough that  the image $\cN^G_H$ of $\cN^H_H$ in 
$\sub(G)/G$ also consists of subgroups with finite
Weyl group. Now we take  $\cV_H^G=\Lambda_{ct}(\cN^G_H)$ to be  the closure under 
cotoral specialization of $\cN^G_H$. The dominant subgroup $H$ does
not quite determine  $\cV^G_H$ (the  choice of $\cN^H_H$
cannot always be made canonical), 
but it does control it, and the block consists of conjugacy classes of
subgroups of the dominant subgroup  $H$.

The existence of the partition of $\fX_G$ into blocks means that the algebraic
model splits as a product 
$$\cA (G)=\cA (G| \cV_1^G)\times \cA (G| \cV_2^G)\times \cdots \times
\cA (G| \cV_n^G). $$
This corresponds to the decomposition
$$\Gspectra \simeq \Gspectra \lr{\cV^G_1}\times
\Gspectra \lr{\cV^G_2}\times \cdots \times \Gspectra \lr{\cV^G_n}$$
into $G$-spectra with the specified geometric isotropy given by the
Burnside ring idempotents supported on the blocks.

In the present paper,  we describe a decomposion of the space
$\fX_{Sp(2)}$ into blocks, and assemble previous  work to give
 a model of rational $Sp(2)$-spectra. There are six 2-dimensional
blocks, each of them dominated by a toral group, so the additional work is in understanding 
fusion and monodromy. Of these, one is dominated by the torus itself, two are
mixed and four are Weyl-finite. There are twenty-six 1-dimensional blocks;
sixteen of them are Weyl-finite and ten are cotoral lines. All but
four are
dominated by a 1-dimensional toral group. The number of 0-dimensional
blocks has not been determined, but 12 of them are dominated by infinite
subgroups. In the sequel \cite{sp2q2}, the exact structure of each block will be
made explicit and used to make calculations.

\subsection{Associated work in preparation}
This paper is one of  a series constructing an algebraic
category $\cA (G)$ for small $G$  and showing it gives an algebraic model for 
rational $G$-spectra. This series gives a concrete 
illustrations of general results in small and accessible 
examples.

The first paper \cite{t2wqalg} describes the group theoretic data
that feeds into the construction of an abelian category $\cA (G)$ for
a toral group $G$ and makes it explicit for toral subgroups of rank 2
connected groups. In a similar vein \cite{gqblocks} shows how
individual blocks look for a general group.

The simplest blocks are those where all groups have finite Weyl
group. The paper \cite{gqwf} deals with these: this covers with 
16 of the non-isolated blocks occurring in $Sp(2)$. 

The  paper \cite{gq1} constructs algebraic models for all relevant 1-dimensional 
blocks: this covers the 10 1-dimensional blocks of cotoral depth 1 for $Sp(2)$.

The other papers have more specific focuses. The
paper \cite{t2wqmixed} constructs algebraic models for
blocks of rank 2 toral groups of mixed type: this covers both
of the two mixed 2-dimensional blocks for $Sp(2)$.

The papers \cite{u2q, su3q} give algebraic models for  $G=U(2)$ and
$G=SU(3)$ respectivley. 

This series is part of a more general programme. A future installment 
dealing  with Noetherian Balmer spectra \cite{AGnoeth} is planned. 
An account of the general nature of the models is in preparation 
\cite{AVmodel}, and the author hopes that this will be the basis of the proof that the 
category of rational $G$-spectra has an algebraic model in general.

\subsection{Organization}
In Section \ref{sec:proper} we summarise the situation for certain
rank 1 groups that occur as subgroups of the rank 2 groups of
interest. In Section \ref{sec:normalizers} we summarise standard methods we
use repeatedly to calculate normalizers.

 In Section
\ref{sec:algorithm} we describe our method for listing the dominant
subgroups $H$ of $G$ and
determining the crude structure of each block. In Section 
\ref{sec:maxrank} we implement this for blocks dominated by 
maximal rank subgroups $H$ of
the groups $G=Sp(1)\times Sp(1)$, $SU(3)$ and $Sp (2)$. In Section 
\ref{sec:regular} we show the only rank 1 regular dominant subgroup $H$
occurs for the local type of  $G=Sp(1)\times Sp(1)$.

For singular rank 1 
dominant subgroups we separate the cases. 
In Section \ref{sec:domA1A1} we describe the dominant subgroups $H$ of rank 1 for 
$G=Sp(1)\times Sp(1)$ and summarize all blocks of subgroups for this 
group $G$.
In Section \ref{sec:domA2} we describe the dominant subgroups $H$ of rank 1 for 
$G=SU(3)$ and summarize all blocks of subgroups for this 
group $G$.
In Section \ref{sec:domC2} we describe the dominant subgroups $H$ of rank 1 for 
$G=Sp(2)$ and summarize all blocks of subgroups for this 
group $G$.

Finally in Section \ref{sec:models} we outline the algebraic models
for rational $G$-spectra for  $G=SU(2)\times SU(2), SU(3)$ and $Sp(2)$.

\subsection{Notation and conventions}
It is worth emphasizing that although many of our ambient groups are
connected, understanding  disconnected groups is essential to our
purpose, so any connectedness assumption will be made explicitly.

Similarly, even when groups are connected,  it is not enough to know
their local type (i.e., their Lie algebra). Nonetheless, the general form of the
spaces of subgroups is very similar for groups of the same local
type.  For the ambient group $G$ we have generally chosen the simply
connected form for definiteness. 

Several of the groups are known by more than one name: generally we
have made a choice, but for $Sp(1)\cong SU(2)$, its roles as a subgroup
of $Sp(2)$ and $SU(3)$ pull in different directions, and we found it
too awkward to be consistent.

\section{Subgroups  of  rank 1 subgroups of $Sp(2)$ }
\label{sec:proper}

In constructing the models $\cA (G)$, it is convenient to proceed
group by group, steadily increasing the complexity of $G$. In fact $\cA (G)$  essentially
contains the models $\cA (H)$ for all subgroups $H$ of $G$, so it
is convenient to have a the models $\cA (H)$ to hand before tackling
$\cA (G)$. The word `essentially' covers two main changes: (i) several 
$H$-conjugacy classes may fuse to form a $G$-conjugacy class and (ii)
the normalizer of a subgroup $K$ in $H$ is a subgroup of the
normalizer in $G$ and hence the $H$-Weyl group $W_H(K)$ is a subgroup
of the $G$-Weyl group $W_G(K)$. These two changes mean that $\cA (H)$
will need to be adapted  to give the corresponding part of
the model in $G$, but the adjustments are secondary in nature.
However the 1-dimensional blocks are all very
standard. We illustrate this by describing $\fX_G$ for a number of
rank 1 groups $G$.

\subsection{The circle group $SO(2)$}
\label{subsec:so2}
The proper subgroups of the circle group form the set $\cC$ of finite
cyclic groups. With the constructible topology
$\fX_{SO(2)}=\sub (SO(2))=\cV^{SO(2)}_{SO(2)}$ is the one-point
compactification of $\cC$. The cotoral relation has finite subgroups cotoral in
$SO(2)$, so $\fX_{SO(2)} $ with the Zariski topology is homeomorphic
to $\spec (\Z)$. We will refer to this as the {\em cotoral line}.

\subsection{The group $O(2)$}
The space $\fX_{O(2)}=\sub(O(2))/O(2)$ can be broken into  two blocks. The 
toral part $\cV^{O(2)}_{SO(2)}$  is 
equal to $\fX_{SO(2)}=\sub(SO(2))$, precisely as in Subsection
\ref{subsec:so2},  since each subgroup of
$SO(2)$ is characteristic. The remainder consists of the 
discrete space  $\mcD$ of conjugacy classes of finite dihedral groups, together with 
$O(2)$ itself as the one point compactification, and in this case
there are no cotoral inclusions so the Zariski topology coincides with
the Hausdorff metric topology. We will refer to this as the {\em flat line}. Altogether we have 
$$\fX_{O(2)}=\cV^{O(2)}_{O(2)}\amalg \cV^{O(2)}_{SO(2)}.$$

\subsection{The group $SO(3)$} 
The space $\fX_{SO(3)}=\sub(SO(3))/SO(3)$ can be broken into 7 blocks. There  
are 5 singleton blocks $\cV^{SO(3)}_{H}$ dominated by $H\in \{SO(3), A_5, \Sigma_4, A_4,
D_4 \}$. There is the block $\cV^{SO(3)}_{SO(2)}$ dominated by the
maximal torus $SO(2)$, and
there is the block $\cV^{SO(3)}_{O(2)}$ dominated by the
normalizer $O(2)$ of the maximal
torus. In summary, we have a partition  
$$\fX_{SO(3)}=\cV^{SO(3)}_{SO(2)}\amalg \cV^{SO(3)}_{O(2)}\amalg 
\cV^{SO(3)}_{SO(3)}\amalg  \cV^{SO(3)}_{A_5}\amalg \cV^{SO(3)}_{\Sigma_4}\amalg \cV^{SO(3)}_{A_4}  \amalg \cV^{SO(3)}_{D_4}  $$
into 7 clopen pieces. We note here that $\cV^{SO(3)}_{O(2)}$ is the
1-point compactification of $\mcD'$ of dihedral subgroups of order
$\geq 6$. The remaining two conjugacy classes from $\mcD$ are treated
separately. In $SO(3)$, the dihedral group $D_2$ is conjugate to $C_2$
so that it appears in $\cV^{SO(3)}_{SO(2)}$. It is convenient to treat
the dihedral group $D_4$ separately in $\cV^{SO(3)}_{D_4}$ because its
Weyl group (namely $D_6$) is larger than that of all larger dihedral
groups (namely $D_2$).

\subsection{The group $Sp(1)$}
Of course $Sp(1)$ is the double cover of $SO(3)$, meaning that
$\fX_{Sp(1)}$ has the same general character as $\fX_{SO(3)}$. In this
case it is homeomorphic: $\fX_{Sp(1)}$ is 
also a disjoint union of 7 blocks, 5 of them singletons, one a cotoral
line, and one the compactification of a countable set. Note that the
homeomorphism is not induced by the quotient map $Sp(1)\lra SO(3)$. 
The dominant subgroups are the double covers of the dominant subgroups of $SO(3)$,
and we write $\At_5, \Sigmat_4, \At_4$ and $\Dt_4$ to highlight the
connection. 

\subsection{Undominated subgroups of $Sp(1)\times C_2$}
We will not give a complete block decomposition of $Sp(1)\times C_2$,
but there is one  counting question that comes up repeatedly. 

We treat subgroups of $T\times C_2$ or  $Pin(2)\times 
C_2$ as dominated (ie counted already) unless they map  to $Q_8$ in 
$Sp(1)$. 

\begin{lemma}
  \label{lem:counting}
  There are 10  conjugacy classes of  finite subgroups of $Sp(1)\times C_2$
  which are not dominated. They fall into the following three types 
  \begin{itemize}
  \item subgroups 
  $F_1^+:=F_1\times 1$ where $F_1$ is the double cover  of $A_5, \Sigma_4, A_4$
  or $D_4$ (four conjugacy classes). The Weyl groups are $1\times C_2, 
  1\times C_2, C_2\times C_2, D_6\times C_2$ respectively. 
\item  the image $F_2^{-}$  of $\{i, \beta\}:F_2 \lra Sp(1)\times C_2$ where $F_2$ is 
  the double cover of $\Sigma_4$ or $D_4$ and $\beta: F\lra C_2$ is 
  nontrivial (two conjugacy classes). The Weyl groups are $1, D_6$ respectively.
  \item  a group $F_3\times C_2$ where $F_3$ is the double cover  of $A_5, \Sigma_4, A_4$
  or $D_4$ (four conjugacy classes).  The Weyl groups are $1, 
  1, C_2, D_6$ respectively. 
  \end{itemize}
\end{lemma}

\begin{proof}
A finite subgroup $F$ is specified by a map  $\{\alpha , \beta\}: F\lra 
Sp(1)\times C_2$. If $\beta$ is trivial, we have a subgroup $F_1$ 
of $Sp(1)$, giving the  4 named conjugacy classes of undominated 
finite subgroups. 

Otherwise $\beta$ is non-trivial and $F':=\ker(\beta)$ is an index 2 
subgroup, and $F'$ is one of the 4 conjugacy clases. 
If we let $F'':=\ker (\alpha)$ we have $F'\cap F''=1$. 

{\em Case 1:  $F''=1$} In this case 
$F$ is specified by $F_2=F'$ together with an epimorphism $\beta: F_2 
\lra C_2$, so $F_2$ is not $\tilde{A}_5$ or $\tilde{A}_4$, and there is 
only one choice if $F_2=\tilde{\Sigma}_4$. Finally if 
$F_2=\tilde{D}_4=Q_8$ there are three maps to $C_2$, but they give the 
same image.

{\em Case 2:  $F''\neq 1$} In this case 
every non-trivial element of $F''$ is of the form $(x, -1)$, so there 
is only one element of this form and $x$ is the central element of 
order 2 in $Sp(1)$. Since the 4 conjugacy classes consist of subgroups 
containing $x$, we deduce $F_3=F'\times C_2$.  
  \end{proof}

\section{Normalizers}
\label{sec:normalizers}
We repeatedly need to calculate normalizers of subgroups. We have
found three methods useful. Actually, they are really the same well known method,
 but it is still useful to lay it out here together with some obvious examples.

\subsection{Normalizers from actions}
Many  of the subgroups we deal with arise as stabilizers of points in some
action. 

\begin{lemma}
If $G$ acts on a space $X$ then 
$gG_x g^{-1}=G_{gx}$ and so  the normalizer of $G_x$ is 
given by 
$$N_G(G_x)=\{ g\in G\st G_{gx}=G_x\}. $$
\end{lemma}

Thus $G_x$ is self-normalizing if $G_x$ determines $x$, but otherwise
$W_G(G_x)$ measures the extent to which it doesn't. 

We illustrate this by taking $G=U(n)$ and $X$ to be derived from
the natural representation $V$ of $G$ on $\C^n$.

\begin{itemize}
\item For $0\neq v\in V$ we have $G_v\cong U(n-1)$. The
  subgroup $G_v$ determines the line generated by $v$ and we have
  $N_{U(n)}(U(n-1))\cong U(1) \times U(n-1)$. 
\item For $0\neq v\in V$ we have $G_{\langle v\rangle}\cong U(1)\times
  U(n-1)$. For $n\neq 2$ the subgroup determines the line generated by
  $v$ so the group is self-normalizing, but if $n=2$ the Weyl group is
  of order 2. 
\item If $U\subseteq V$ is a subspace of dimension $d$ then 
  $$G_U\cong U(d)\times U(n-d). $$
  If $n\neq 2d$ then the subgroup determines $U$ and the group is
  self-normalizing. If $n=2d$ the Weyl group is of order 2. 
\item If 
$$U_{\bullet}=\left( 
0=U_0\subseteq U_1\subseteq U_2\subseteq \cdots \subset U_s=V\right)$$ 
is a flag, we write $\Ub_i$ for the orthogonal complement of $U_{i-1}$
in $U_i$, and note that for $G=U(n)$, preserving a flag is equivalent
to preserving the  direct sum 
decomposition 
$$V=\Ub_1\oplus \Ub_2\oplus \cdots \oplus \Ub_s, $$
so that if  $d_i=\dim_{\C}(\Ub_i)$ then 
$$G_{U_{\bullet}}=U(d_1)\times U(d_2)\times \cdots \times U(d_s). $$
We see that the Weyl group is a product of symmetric groups $\Sigma_s$
corresponding to the multiplicities of the dimensions. 
  
In particular the maximal tori are the stabilizers of complete flags
and the Weyl group of the maximal torus in $U(n)$ is $\Sigma_n$
\end{itemize}

\subsection{Normalizers from representations}
If $\rho : H \lra U(n)$ is an $n$-dimensional representation of  $H$ we can consider
the image $\rho (H)$.

\begin{lemma}
  \label{lem:norminun}
If $H$ is connected and $\rho$ is simple then the identity component
of $N_{U(n)}(\rho (H))$ lies in  $H\cdot ZU(n)$.
\end{lemma}

\begin{proof}
Let $N=N_{U(n)}(\rho H)$.
First note that conjugation defines a map $N\lra \aut (\rho
(H))$. Since $\out (\rho (H))$ is discrete, the 
image of $N_e$ lies in the inner automorphisms. Hence for each $g\in N_e$ there
is an $h\in H $ so that $gh$ centralizes $\rho (H)$. 
By Schur's Lemma, this means that $gh$ acts as a scalar. 
\end{proof}

In many cases, $\out (\rho (H))$ is small and well known, so that we
get information about the component group of the normalizer. 
For example $SU(2)$ has trivial outer automorphism group, so 
$$N_{U(n)}(\rho (SU(2))=\rho (SU(2))\cdot ZU(n), $$
and this also gives the answer when the representation extends from a
subgroup of $U(n)$. For example it applies for $V_4$ with $G=Sp(2)$,
to show the normalizer of $V_4(SU(2))$ is selfnormalizing in $Sp(2)$ 
(since $V_4(SU(2))$ already contains $\pm I=ZU(4)\cap Sp(2)$).

\subsection{Normalizers from Centralizers}
In its crudest form, we have the following  general
principle. This and variations  are used numerous times.

  \begin{lemma}
\label{lem:gennormalizer}
If $H$ is connected then 
$N_G(H)\subseteq N_G(T(H))\cdot H$. 


If $H$ is of rank 1 then $\out (T(H))$ is of order 2, so that 
if $H$ is not a torus, $H$ contains an element mapping onto $\out 
(T(H))$
and 
$$N_G(H)\subseteq Z_G(T(H))\cdot H .$$
  \end{lemma}

  \begin{proof}
If $n$ normalizes $H$ then $T(H)^n$ is a maximal torus, so that there 
is an $h\in H$ with $T(H)^{nh}=T(H)$. Thus $n=(nh)h^{-1}\in 
N_G(T(H))\cdot H$. 

If $T(H)$ is a circle and the Weyl group of $H$ is of order 2, there 
is another element $h'\in H$ so that $nhh'$ centralizes $T(H)$
    \end{proof}

When we can reduce to centralizers, there is a well known formula for
the Lie algebra, determining the identity component. 
We suppose $G$ is a compact connected Lie group with maximal torus
$T$, and the conventions of \cite{BtD}.  Amongst the real roots
$R\subseteq LT^*$, we have chosen positive roots $R_+$ and simple roots $S$.
Thus 
$$LG_{\C}=L_0\oplus \bigoplus_{\alpha \in R} L_{\alpha}, $$
and 
$$LG=M_0\oplus \bigoplus_{\alpha \in R_+} M_{\alpha}, $$
where 
$$M_{\alpha}=(L_{\alpha}\oplus L_{-\alpha})\cap LG=M_{-\alpha}.$$

If $t$ is an element of the maximal torus $T$
then
$$LZ(t)=M_0\oplus \bigoplus_{\alpha\in R(t)} M_{\alpha}$$
where
$$R(t)=\{ \alpha \in R_+\st t\in \ker (\alpha)\}. $$

When combined with the obvious fact that
$$Z(t)^g=Z(t^g)$$
this enables us to calculate normalizers in many cases.

\section{Listing dominant subgroups}
\label{sec:algorithm}
We describe here the strategy for listing conjugacy classes of
 dominant subgroups (the strategy is proved effective in general in 
\cite{gqblocks}, but we will verify it is effective in our examples by
following through the process). In fact
it is a strategy for listing all subgroups, but they are listed in an
order which means we can easily spot dominated subgroups and omit
them, to be left with the
finite list of dominant conjugacy classes. 

The idea is that we start with big subgroups and work down. Here the size if a
subgroup is the triple 
$$(\rank(H), \dim (H), |H_d|). $$
We arrange the subgroups in reverse lexicographical order,
 so that we start with subgroups of rank
equal to that of $G$ and work down. For each rank we first determine
those of largest dimension, and for each rank and dimension we find
all subgroups, and then arrange them in decreasing order of the 
 number of components. If  there is no bound on $|H_d|$, there will be
 a limit point of the groups with the same rank and dimension, and all
 but finitely many of the groups will be dominated by a subgroup of
 larger dimension.  

The reason for taking things in this order is that we can ignore
subgroups that are dominated, since the dominating subgroups will 
 have  been considered earlier. A block $\cV^G_H$ is 
formed by choosing a neighbourhood $\cN$ of $H$ in the space of conjugacy
classes of subgroups with finite Weyl group and then taking the
cotoral down-closure. Thus dominated subgroups $K$ are cotoral in
subgroups close to a dominant subgroup $H$. It is obvious that  cotoral
overgroups have larger rank and dimension, and convergent sequences
are either eventually constant or tend to a subgroup of larger
dimension. Thus any candidates for dominators of $K$ will have
occurred earlier. 

On practical grounds, the reverse lexicographic order also has the benefit of giving  us the
easiest task first. In fact we will not discuss here the
classification of the maximal finite subgroups (i.e., those of size $(0,0,n)$) at all; this
is reasonable since they are topologically and cotorally isolated, so
their contribution to the algebraic model is simply the category of
$\Q W_G(H)$-modules; certainly not a trivial matter, but a purely
algebraic one.

\subsection{Making the list}
\label{subsec:listingdom}
Given the rank  $s$ of $H$, the classification of dominant subgroups 
proceeds as follows. 
Since every element of a connected group lies in a maximal torus, the 
exponential map is surjective on $H_e$. Thus  $H_e$ is determined by the 
subalgebra $LH$ of $LG$. Thus the potential identity components $H_e$
are determined by the theory of Lie algebras. 

For groups of rank 2 this is relatively easy since subgroups of 
maximal rank are particularly easy (see Section \ref{sec:maxrank}), 
and connected groups of rank 1 are either of local type $SU(2)$
(classified by `nilpotent orbits')  or else circles (classified by 
understanding the usual root data). As mentioned above, groups of 
rank 0 (i.e., finite subgroups) are harder and not dealt with here.

Given a fixed connected subgroup $H_e$, we seek all subgroups $H$ with this 
 identity component. Subgroups $H$ correspond bijectively to  finite subgroups 
 $H_d$ of $W_G(H_e)$.  Subgroups $H$ thus correspond to pairs $(H_e, 
 H_d)$. Note that we have inclusions 
$$\xymatrix{
H_e\rto & H\rto&G\\
N_{H_e}(TH)\rto\uto  &N_{H}(TH)\rto \uto &N_{G}(TH)\uto \\ 
TH\uto&& 
}$$
Thus $W_{H_e}(TH)\subseteq W_H(TH)\subseteq W_G(TH)$
and we may consider the $W_G(TH)$-module $H_1(TH)$ as a module over 
$W_{H}(TH)$ or $W_{H_e}(TH)$  by restriction.

\subsection{Block tables}
The outcome for each group $G$ 
is  a Final Block Table with five columns
(See Section \ref{sec:domA1A1} for 
$G=SU(2)\times SU(2)$, Section \ref{sec:domA2} for $G=SU(3)$ and 
Section \ref{sec:domC2} for $G=Sp(2)$). 
Each row 
corresponds to the conjugacy class of a dominant subgroup $H$. 
The first entry will be the pair $(H_e, H_d)$. 
The second entry is the dimension  of $\cV^G_H$ in the form 
$a+b$ (see below). The third entry is $W_G(H)$. The fourth and fifth 
columns are pointers to further details. The fourth column either
states that $H$ is toral or gives the connected group $\Hhat$ in which 
a block dominated by $H$  first occurs. The fifth column points
to a paper providing full analysis of the model $\cA(\Hhat,
\cV^{\Hhat}_H)$ (here 
[gqwf]=\cite{gqwf}, 
[gq1]=\cite{gq1}, 
[mixed]=\cite{t2wqmixed}, 
[gqtoral]=\cite{gtoralq}, 
[Discrete]=\cite{gfreeq2,spcgq}).

\subsection{Describing the block of a dominant subgroup}
To understand the space of conjugacy classes dominated by $H$ 
we suppose $H_e=\Sigma \times_Z T_z$, where $\Sigma$ is semisimple and
$Z$ is a finite central subgroup. Since $T_z$ is the identity
component of the centre of the identity component, it is characteristic
and $\Lambda_0(H):=H_1(T_z)$ can be viewed as an integral
representation of $H_d$. As described in \cite{gqblocks}, 
the structure of $\cV^G_H$ is largely
determined by this integral representation, and in particular 
the dimension  is the number of simple factors 
of the $H_d$-representation $\Lambda_0(H)\tensor \Q$.

 Up to a finite ramified cover, this is the product 
 of the poset $\sub(T^a)$ (closed subgroups of an $a$-torus)
 and a profinite space of dimension $b$ where 
$a$ is the number of trivial summands in $\Lambda_0(H)\tensor \Q$ and 
$b$ is the number of non-trivial simple modules: the dimension in this
case is recorded as $a+b$. The  $b$-dimensional profinite space 
is in turn a product of factors corresponding to the isotypical 
pieces. 

\subsection{The fundamental lattice}
We note that 
$$TH=T(\Sigma\times_Z T_z)=T(\Sigma) \times_Z T_z$$ 
gives
$$1\lra T_z\lra TH\lra T\Sigma/Z\lra 1.$$
Looking at the fundamental group we see this  induces a short exact sequence
$$0\lra H_1(T_z)\lra H_1(TH)\lra H_1(T\Sigma/Z)\lra 0.$$
By Lemma \ref{lem:Tzsub}, this is an exact sequence of $W_H(TH)$-modules.

The group $V=W_{H_e}(TH)$ is a subgroup
of $W_H(TH)$ and $\Lambda_0(H)=H_1(T_z)$ is $V$-fixed.

Finally, the topology on $\cV^H_H$ is determined by the $H_d$-module
$\Lambda_0(H)=H_1(T_z)$ \cite{gqblocks}, which by Lemmas
\ref{lem:Hdepi} and \ref{lem:Tzsub} below
 is determined by the $W_H(TH)$ action
on $H_1(TH)$. The simplest case is when $H_e$ is semisimple, since then
 $T_z=1$, $\Lambda_0(H)=0$ and  $\cV^G_H$ is a point.

\subsection{The Weyl action and the central action}

We have  a disconnected subgroup $H$ with maximal torus $TH$, giving a diagram 
$$\xymatrix{
1\rto & TH \rto \dto & N_H(TH)\rto \dto & W_H(TH)\rto \dto^{\lambda} &1\\
1\rto & H_e \rto \   & H  \rto                  & H_d\rto &1 
}$$
Now we consider the identity component of 
the centre of the identity component $T_zH\subseteq TH$.  By 
definition it is characteristic in $H$. To simplify notation we write 
$T_z=T_zH$.

\begin{lemma}
  \label{lem:Hdepi}
The map $\lambda: W_H(TH)\lra H_d$ is surjective. 
\end{lemma}

\begin{proof}
If $g\in H$ then $TH^g$ is a maximal torus in $H_e$, and hence there 
is an $h\in H_e$ with $TH^{gh}=TH$. Evidently $g$ and $gh$ give the 
same element of $H_d$. 
\end{proof}

\begin{lemma}
  \label{lem:Tzsub}
The map $T_z\lra TH$ induces a $W_H(TH)$-equivariant map 
$$H_1(T_z)\lra H_1(TH). $$
\end{lemma}

\begin{proof}
  Any element of $H$ normalizes $T_z$ and hence for any $n\in N_H(TH)$, 
  we have a commutative square 
$$\xymatrix{
T_z\rto \dto_{c_n}&TH\dto^{c_n}\\
T_z\rto & TH. 
}$$
This gives a commutative square 
$$\xymatrix{
H_1(T_z)\rto \dto_{\lambda n_*}&H_1(TH)\dto^{n_*}\\
H_1(T_z)\rto & H_1( TH). 
}$$
\end{proof}

\section{Maximal rank blocks}
\label{sec:maxrank}

In this section we list the dominant subgroups containing a maximal
torus. Since the dimension of a block is bounded by the rank of the
dominant subgroup, the most complicated blocks will occur here. 
On the other hand, the analysis is rather straightforward 
from  the root system and Weyl group of $G$. 

We describe the strategy in Subsections \ref{subsec:working} and
\ref{subsec:wf} and then discuss $A_1\times A_1, A_2$ and $ C_2$ in
turn.  It is perfectly practical to deal with $G_2$ as well, but this
is deferred. 

\subsection{Completing the block tables}
\label{subsec:working}
We explained above that for each $G$, we will record the Final Block Table
with five columns to describe the blocks that occur. In this section,
as a step towards the Final  Table,  we will make a Working  Table for
the maximal rank blocks. In Sections \ref{sec:domA1A1},
\ref{sec:domA2} and \ref{sec:domC2} remove working and reorder the 
columns to give the Final  Table. 
Each row in the Working Table corresponds to the local type $H_e$ of a 
potential dominant subgroup $H$. 
The essential things for the Final Table are 
 $H_e, H_d$ and the shape of the block  $\cV^H_H$:
 these are the 3rd, 4th and 5th columns in the 
Working Table. The Working Table is arranged as we fill it in along
each row, so that the 1st and 2nd 
columns are steps along the way to the 3rd.

The 1st column lists the Weyl group $W_{H_e}(TH)$. When $H$ is of 
maximal rank, $TH=TG$, this is a subgroup of 
the finite group $W_G(TG)$. The 1st column lists all subgroups of $W_G(TG)$, as they 
are candidates for $W_{H_e}(TH)$,  and the 2nd records that some are 
quickly  disqualified since they are not generated by reflections. The 
1st and 2nd columns then determine the local type of $H$, which is 
recorded in the 3rd column.  Since $G$ is given, the local type of $H_e$ determines 
the subgroup $H_e$ itself.  

Next, one must  calculate $W_G(H_e)$. The finite 
subgroups which are maximal in the sense that they are not dominated
by  a subgroup already listed as dominant are then listed in the 4th column. 

If $H$ is of maximal rank, $W_G(H_e)$ is 
itself finite (Lemma \ref{lem:WGHsubquotWGT}), so each of its subgroups are maximal in this sense. 
Each maximal  finite subgroup  gives the component group 
$H_d$ for a subgroup $H$ with identity component $H_e$. 
The pair $(H_e,H_d)$ (with the action of $H_d$ on $H_e$ and the 
relation of the subgroups to $G$) determines $H$. 

Having identified the pair $(H_e,H_d)$, the module $\Lambda_0(H)$ is 
determined. In the maximal rank case we can say that $TG=TH$ and 
 we have the $W_G(TG)$-module $H_1(TG)$, which restricts to a 
 $W_H(TG)$-module.  Furthermore, by Lemma 
\ref{lem:Tzsub},  the subgroup $H_1(T_z)$ is $W_H(TH)$-submodule, with 
the action via $\lambda : W_H(TH)\lra H_d$. Since $\lambda$ is an 
epimorphism by Lemma \ref{lem:Hdepi}, the $H_d$ action is 
determined. 

Given this data, we can deduce the form of $\cV^H_H$. This is 
described in more detail in \cite{gqblocks}, but we tabulate a crude 
version.  Up to a finite ramified cover, this is the product 
of a poset $\sub(T^a)$ and a profinite space of dimension $b$ where 
$a$ is the number of trivial summands in $\Lambda_0(H)\tensor \Q$ and 
$b$ is the number of non-trivial simple modules.  The profinite space 
is in turn almost a product of factors corresponding to the isotypical 
pieces. In the 5th Column, we then record the type $a+b$ of $\cV^H_H$.

\subsection{Maximal rank subgroups are Weyl-finite}
\label{subsec:wf}
The two main simplifications of the maximal rank case are firstly the
simple reduction to the combinatorics of the root system and secondly
the finiteness of Weyl groups. 

\begin{lemma}
\label{lem:WGHsubquotWGT}
If $H$ is of maximal rank then $W_G(H)$ is a subquotient of the Weyl 
group $W_G(T)$, and hence in particular finite. 
\end{lemma}

\begin{proof}
The identity component of $H$ is characteristic in $H$, 
so  $H_e\subseteq H\subseteq N_G(H)\subseteq N_G(H_e)$, and 
it suffices to deal with the case that $H$ is connected. 

If $g$ normalizes $H$, then since all maximal tori in $H$ are 
conjugate, we deduce there is an $h\in H$ so that $gh\in 
N_G(T)$. Since the Weyl group is finite, $H$ is of finite index in $N_G(T)$. 
  \end{proof}

First we consider the identity component $H_e$. The Lie algebra $LH$
is a subalgebra of $LG$ with $M^H_0=LT=M^G_0$ and $R^H\subseteq
R^G=R$ and we may suppose $R^H_+\subseteq R^G_+=R_+$.
Altogether, this shows that a choice of $H_e$
corresponds to a subgroup of
$W_G(T)$ generated by
reflections, and  the identity component of a maximal rank subgroup is
determined by combinatorics.

Lemma \ref{lem:WGHsubquotWGT} shows that $W_G(H)$ is finite, and  a subquotient
of $W_G(T)$. 

\subsection{Local type $A_1\times A_1=D_2$; the group $G$ is finitely
  covered by $SU(2)\times
  SU(2)$}
 The Weyl group $D_4=W_G(T)$ is generated by reflections in two
 perpendicular lines (the $x$ axis and the $y$ axis).
 There are thus 5 subgroups of the Weyl group, 4 of them reflection
 subgroups. Each reflection group corresponds to a connected subgroup
 $H_e$. The possible local types are thus $SU(2)\times SU(2), T\times 
SU(2), SU(2)\times T $ or $T\times T$. The actual subgroups $H_e$  are 
determined by $G$. We will tabulate the possibilities assuming $G$ is 
simply connected, but up to homeomorphism the blocks are the 
same in the central quotient.

$$
\begin{array}{c|cccc}
W_{H_e}(T)  &\mbox{Reflection?}&\mbox{$H_e$ Local type}&H_d&\cV^G_H\\
    \hline 
  D_4&      y& SU(2)\times SU(2) &1& 0\\
  C^x_2&  y&SU(2)\times T &1, C_2&1+0, 0+1\\
  C^y_2&  y&T\times SU(2) &1, C_2&1+0, 0+1\\
   C^\Delta_2&  n&- &-&-\\
   1&y&T\times T& 1, C_2^x, C_2^y, C_2^\Delta, D_4&2+0, 1+1, 1+1, 0+2,
                                                  0+1 \\
  \hline 
  \end{array}$$

In reading the table,  Weyl groups arise in several ways: we discuss the groups 
$H$ row by row (according to $W_{H_e}(T)\subseteq D_4$), in each case 
considering the finite subgroup $H_d \subseteq W_G(H_e)$. 

The first row corresponds to $W_{H_e}(T)=W_G(T)=D_4$. Its 
Weyl group $W_G(H_e)$ is clearly trivial.

The second row corresponds to $W_{H_e}(T)=C_2^x$.  The  
Weyl group $W_G(H_e)$ is of order 2. This follows
from the fact that we are seeking the normalizer of $G_{\langle 
  v\rangle}$ for a non-zero vector: for example, the normalizer of 
$T\times SU(2)$ is $N_{SU(2)}(T)\times SU(2)$. This gives two
possibilities for $H_d$. When $H_d=1$ we obtain $H=T\times SU(2)$
with dominated block consisting of subgroups $C\times SU(2)$, and
therefore homeomorphic to $\sub(T)$ (and hence of dimension $1+0$).
When $H_d=C_2$ the action of $C_2$ on $T$ is non-trivial, and  
we obtain $H=Pin(2)\times SU(2)$. In this case, the dominated block
consists of subgroups $K\times SU(2)$ with $K\subseteq Pin (2)$ and is
therefore homeomorphic to  the 
space of quaternionic subgroups of $Pin(2)$ (and hence of dimension $0+1$).

The third row, corresponding to $W_{H_e}(T)=C_2^y$  is precisely
similar to the second with roles of $x$ and $y$ exchanged. 

The fourth row $C_2^{\Delta}$ is not a reflection subgroup so does not 
correspond to a subgroup $H$.

Finally, the last row corresponds to $W_{H_e}(T)=1$, so that 
$H_e=T$. Its Weyl group $W_G(T)=D_4$, giving 5 different conjugacy 
classes of subgroup $H_d$.

\subsection{Local type $A_2$; the group $G$ is finitely covered by $SU(3)$}
 The Weyl group  $D_6=W_G(T)$, generated by two reflections at an 
angle of $\pi/3$. There are thus four conjugacy classes of subgroups. 
The trivial subgroup of $D_6$  (giving $H$ of local type $T$), 
the three conjugate subgroups of order 2 (giving $H$ of local type 
$T\times SU(2)$),  the subgroup of order 3, which is not a reflection
group,  and the group $D_6$ itself (giving $H=G$).
The actual types are  determined by $G$. We will continue the
discussion assuming we have 
the simply connected form, but up to homeomorphism the blocks are the 
same in the central quotient.

$$
\begin{array}{c|cccc}
  W_{H_e}(T)&\mbox{Reflection?}&\mbox{Local type}&H_d&\cV^G_H\\
    \hline 
  D_6& y& SU(3)&1& 0\\
  C_3&  n&- &-&-\\
  C_2&y&U(2)& 1&1+0\\
  1&y&T^2& 1, C_2, C_3, D_6&2+0, 1+1, 0+1,0+1 \\
                 \hline 
  \end{array}$$

In reading the table,  Weyl groups arise in several ways: we discuss the groups 
$H$ row by row (according to $W_{H_e}(T)\subseteq D_6$), in each case 
considering the finite subgroup $H_d \subseteq W_G(H_e)$. 

The first row corresponds to $W_{H_e}(T)=W_G(T)=D_6$. Its 
Weyl group $W_G(H_e)$ is clearly trivial. 

The second  row $C_3$ is not a reflection subgroup so does not 
correspond to a subgroup $H$.

The third row corresponds to $W_{H_e}(T)=C_2$, so that $H_e$ is of
local type $T\times SU(2)$. Considering 3 dimensional representations
of $T\times U(2)$ we see that the only connected subgroups of $SU(3)$
of this local type are $H_e\cong U(2)$, all of them conjugate and
determined by the central circle. Up to conjugacy $H_e$ is  
 the centralizer of $\diag (\lambda, \lambda, 
\lambda^{-1})$. This is uniquely specified by a choice of line in  the 
natural representation.  From our discussion of normalizers we see
that the Weyl group $W_{SU(3)}(U(2))$ is trivial, so $H_d=1$. 
For the subgroups dominated by $U(2)$, we have the cotoral line 
$\sub(T)$ of subgroups between $SU(2)$ and $U(2)$ (dimension $1+0$).

Finally, the last row corresponds to $W_{H_e}(T)=1$, so that 
$H_e=T$. Its Weyl group $W_G(T)=D_6$, giving 4 different conjugacy 
classes of subgroup $H_d$.  To identify the structure of the block 
dominated by $H$ we consider the action of $H_d$ on $\Lambda_0(H)=H_1(T)$. 
The representation $\Lambda_0(H)\tensor \Q$ is the natural representation 
of $D_6$ on the rational plane.  The restrictions to the subgroups $H_d$ is 
easily seen: it is $\Q \oplus \Q$  when $H_d$ is the trivial group (of
dimension $2+0$), it is  $\Q\oplus \Qt$ when $H_d=C_2$ (giving
dimension $1+1$), it is a 2-dimensional simple 
representation when $H_d$ is of order 3 or 6 (giving dimension $0+1$
in both cases).

\subsection{Local type $C_2$; the group $G$ is finitely covered by $Sp(2)$}
 The Weyl group  $D_8=W_G(T)$, is generated by two reflections at an 
angle of $\pi/4$. There are 6 conjugacy classes of
reflection subgroups represented by $D_8, V, V', X, X', 1$ where $V$
(a Klein 4-group) is generated by reflections in the long roots, 
$V'$ (another Klein 4-group) is generated by the reflections in the
short roots, $X$ (of order 2) is generated by a reflection 
in a long root and $X'$ (of order 2) is generated by a reflection 
in a short root. There are two other subgroups generated by rotations
which are not reflection groups.  This reflection subgroup  determines the local type of
$H_e$. The actual types are  determined by $G$. We will continue the
discussion assuming we have  the simply connected form, but up to homeomorphism the blocks are the same in the central quotient.

$$
\begin{array}{c|cccc}
 W_{H_e}(T) &\mbox{Reflection?}&\mbox{Local type}&H_d&\cV^G_H\\
    \hline 
  D_8& y& Sp(2)&1& 0\\
  V&  y&Sp(1)\times Sp(1) &1, C_2&0, 0\\
  C_4=\langle \omega\rangle&n& -&-&-\\
  V'&y&Sp(1)\times Sp(1)& 1, C_2&0, 0\\
  X&y&T\times Sp(1) &1,C_2&1+0, 0+1\\
  C_2=\langle \omega^2\rangle &n&-&-&-\\
  X'&y&T\times Sp(1) &1,C_2&1+0, 0+1\\
          1&y&T& 1&2+0\\
    &&& X, C_2, X'&1+1, 0+2, 1+1\\
    &&& V, C_4, V'&             0+2,0+1, 0+2\\
    &&&  D_8&0+1\\
                 \hline 
  \end{array}$$

In reading the table,  Weyl groups arise in several ways: we discuss the groups 
$H$ row by row (according to $W_{H_e}(T)\subseteq D_8$), in each case 
considering the subgroup $H_d \subseteq W_G(H_e)$. 

The first row corresponds to $W_{H_e}(T)=W_G(T)=D_8$. Its 
Weyl group $W_G(H_e)$ is clearly trivial.

 The second row corresponds to $W_{H_e}(T)=V$, generated by the reflections given by the long
 roots. The group is one fixing two perpendicular quaternionic lines,
 and corresponds to  the representation $V_2\tensor \eps \oplus \eps \tensor V_2$. This is 
uniquely specified by a choice of two perpendicular lines, which may 
be exchanged, so that $W_G(H_e)=C_2$ giving two possibilities for
$H_d$. However in either case $\Lambda_0(H)$ is trivial so the 
block is a point. 

The third  row $C_4$ is not a reflection subgroup so does not 
correspond to a subgroup $H$. 

 The fourth row corresponds to $W_{H_e}(T)=V'$, generated by the reflections given by the short
 roots. We may consider the image in $SU(4)$, where it is given by the representation
$V_2\tensor  V_2$. The centre of $Sp(2)$ already
lies inside the subgroup, so as in Lemma \ref{lem:norminun} the only question
is whether the outer automorphism is realised in the group
$Sp(2)$. Since the tensor product is external, the exchange of tensor
factors  is realized and $W_G(H_e)=C_2$ giving two possibilities for
$H_d$. Once again, in either case $\Lambda_0(H)$ is trivial so the 
block is a point.

For the two subgroups of order 2 we use centralizers. 
The centralizer $Z(t)$ of an element $t\in T$ has $LZ(t)$ being $LT$
together with $M_{\alpha}$ for $t\in \ker (\alpha)$. Since
$Z(t)^g=Z(t^g)$, with care about components, this enables us to 
calculate normalizers. 

The fifth row corresponds to $W_{H_e}(T)=X$, generated by reflection from a 
single long root, and it is in the kernel of the other long root. 
 We think of $G'=SU(2)\times SU(2)$ and $ H_e=T\times SU(2)$; note that 
the group is the one centralizing a circle, which is to say
stabilizing a subspace $\langle v\rangle \oplus 0$  in $V_2\oplus V_2$. This is 
uniquely specified by the complex line $\langle v\rangle$.  Thus the normalizer 
$N_G(H_e)=N_{SU(2)}(T)\times SU(2)$ and the Weyl group is of order 2.  
If $H_d$ is of order 1, the block is a cotoral line $\sub (T)$ (of
dimension $1+0$) and if 
$H_d=C_2$ we the block corresponds to the quaternionic subgroups of
$Pin (2)$ (of dimension $0+1$).

The sixth row $\langle \omega^2\rangle$ is not a reflection group so
does not correspond to a subgroup $H$.

The seventh row corresponds to $W_{H_e}(T)=X'$, generated by reflection from a 
single short root, and it is the kernel of the other short root. If
$H_e=Z(t)_e$ then we only have $H_e^g=H_e$ if $Z(t)_e=Z(t^g)$, which
only  happens when $t^g$ is $t$ or $t^{-1}$ so that 
in $N_G(T)/T$ we find $g$ gives reflection in a short root.  Thus
$W_G(H_e)$ is of order 2, giving two possibilities for $H_d$. 
If $H_d$ is of order 1, we have the cotoral line $\sub (T)$ (of
dimension $1+0$) and if 
$H_d=C_2$ we have the quaternionic subgroups of $Pin (2)$ (of
dimension $0+1$).

Finally, the last case corresponds to $W_{H_e}(T)=1$ so that $H_e=T$. 
 Its Weyl group $W_G(T)=D_8$, giving 8 different conjugacy 
classes of subgroup $H_d$, and we have displayed this on four rows
according to the order of $H_d$.  To identify the structure of the block 
dominated by $H$ we consider the action of $H_d$ on $\Lambda_0(H)=H_1(T)$. 
The representation $\Lambda_0(H)\tensor \Q$ is the natural representation 
of $D_8$ on the rational plane.  The restrictions to the subgroups $H_d$ is 
It is $\Q \oplus \Q$  when $H_d$ is the trivial group (giving
dimension $2+0$).  It is 
$\Q\oplus \Qt$ when $H_d$ is a reflection group of order 2 (giving
dimension $1+1$), and it is the sum $\Qt\oplus \Qt$ when $H_d$ is the
rotation group of order 2 (giving dimension $0+2$). 
For $V,V'$ it is the sum of two non-isomorphic simples (giving
dimension $0+2)$. Finally when 
$H_d=C_4$ or $D_8$,  it is a 2-dimensional simple 
representation (giving dimension $0+1$).

\section{Regular subgroups} 
\label{sec:regular}
We turn to the case that $H_e$ is of rank 1, and start with the case 
that $TH$ is regular. For convenience we choose the maximal tori so that $TH\subseteq  
TG$.

A variation of the Lemma \ref{lem:WGHsubquotWGT} applies here too. 

\begin{lemma}
(i)  If $TH$ contains a regular element  then 
  $N_G(TH)\subseteq N_G(TG)$. 

(ii)   If $H$ contains a regular element then 
  $N_G(H)\subseteq H\cdot N_G(TG)$. 
\end{lemma}

\begin{proof}
We note that $H_e \subseteq H\subseteq N_G(H)\subseteq N_G(H_e)$
so it suffices to deal with the case that $H$ is connected. 

If $g$ normalizes $H$, then since all maximal tori in $H$ are 
conjugate, we deduce there is an $h\in H$ so that $gh\in 
N_G(TH)$. Since $H$ contains a regular element so does $TH$ and 
hence $N_G(TH)\subseteq N_G(TG)$. 
  \end{proof}

\begin{prop}
\label{prop:noregcircles}
If $G$ is semisimple of rank 2, there are no dominant rank 1 subgroups
consisting of regular elements unless $G$ is of local type
$A_1\times A_1$ and $H_e$ is the diagonal copy of $A_1$.
\end{prop}

\begin{proof} 
If $H_e=TH$ is regular, then $N_G(TH)$ is generated by  elements of $N_G(TG)$
preserving $TH$. But the reflection corresponding to a root $\alpha$
 fixes the orthogonal hyperplane, which corresponds to the kernel of
 $\alpha$ and is this would mean $TH$ is singular. This means $H=H_e$, but this is cotoral 
in $TG$ and so is dominated. 

Now suppose $H_e$ is  of local type $SU(2)$. If
$TH\subseteq TG$, this means that each root $\alpha$ of $G$ would
have to restrict either to 0 or to $\pm \beta$, where $\beta $ is the
positive root of $H$. Inspection of the root system diagram shows that
if the local type is $A_2, B_2$ or $G_2$, the only lines $TH$ in $TG$
that have this property are perpendicular to some $\alpha$ and
therefore singular. On the other hand, if $G$ is of local type
$A_1\times A_1$, there are two regular lines with this property. They 
are perpendicular to each other and correspond to the
diagonal. 
\end{proof}

\section{Dominant subgroups for $SU(2)\times SU(2)$}
\label{sec:domA1A1}
We complete the discussion of $G=SU(2)\times SU(2)$ by considering 
singular subgroups of rank 1. The section ends with a complete summary 
of all blocks.

\subsection{The regular diagonal subgroup}
Proposition \ref{prop:noregcircles} showed that the only regular subgroup $SU(2)$ is the
diagonal one. 

Conjugation by $(g_1,g_2)$ takes $(x,x)$ to $(x^{g_1},x^{g_2})$, which
is diagonal only if $g_1g_2^{-1}$ centralizes $x$.  The diagonal
subgroup  is therefore normalized by pairs with $g_1g_2^{-1}$ central,
and hence the Weyl group of the diagonal group is of order 2.

 \subsection{Singular subgroups}
Finally we turn to the case that $H_e$ is of rank 1, and  $TH$ is 
singular. For convenience we choose the maximal tori so that $TH\subseteq  TG$. 

The identity component of a singular subgroup of rank 1 is either a 
torus or locally $SU(2)$. 

In the latter case we have a single faithful 2-dimensional
representation (the identity) so there are three subgroups $SU(2)$, 
namely $1\times SU(2)$, $SU(2)\times 1$ and the regular diagonal. 

The subgroup $ 1\times SU(2)$ has normalizer $G$ and Weyl group
  $SU(2)$. This therefore gives four new 
dominant subgroups, corresponding to the maximal groups $\tilde{A}_5,
\tilde{\Sigma}_4, \tilde{A}_4, \tilde{D}_4$. The subgroup $SU(2)\times 1$ 
is exactly parallel. 

If the identity component is a singular circle it is $1\times T$ or
$T\times 1$. These behave just the same as each other. The normalizer
of the first is $SU(2)\times NT$ with Weyl group $SU(2)\times C_2$. 
In each case there are 10 possibilities, enumerated in Lemma \ref{lem:counting}

\subsection{$G=SU(2)\times SU(2)$ summary table}
$$\begin{array}{cl|cccc|}
\mathrm{rk}&(H_e, F)&\dim (\cV^G_{(H_e,F)})&W_G(H)&\hat{H}& \\
\hline 
2&(G,1)&0&1&G&\mathrm{Discrete}\\
2&(SU(2)\times T,C_2)&0+1&1&G&\gqwf\\
2&(SU(2)\times T,1)&1+0&C_2&SU(2)\times T&\gqone\\
2&(T\times SU(2),C_2)&0+1&1&G&\gqwf\\
2&(T\times SU(2),1)&1+0&C_2&SU(2)\times T&\gqone\\
2&(T^2, D_4) &0+2&1&\toral&\gqwf\\
2&(T^2, C_2^{\Delta}) &0+2&C_2&\toral&\gqwf\\
2&(T^2, C_2^x) &1+1&C_2&\toral &\mix\\
2&(T^2, C_2^y) &1+1&C_2&\toral&\mix\\
2&(T^2, C_1) &2+0&D_4&\mbox{torus} &\gqtoral\\
\hline 
1&(\Delta SU(2),C_2)&0&1&G&\mathrm{Discrete}\\
1&(\Delta SU(2),1)&0&C_2&SU(2)&\mathrm{Discrete}\\
1&(1\times SU(2), \At_5)&0&1&G&\discrete\\
1&(1\times SU(2), \Sigmat_4)&0&1&G&\discrete\\ 
1&(1\times SU(2), \At_4)&0&1&G&\discrete\\ 
1&(1\times SU(2), \Dt_4)&0&1&G&\discrete\\ 
1&(SU(2)\times 1, \At_5)&0&1&G&\discrete\\
1&(SU(2)\times 1, \Sigmat_4)&0&1&G&\discrete\\ 
1&( SU(2)\times 1, \At_4)&0&1&G&\discrete\\ 
1&(SU(2)\times 1, \Dt_4)&0&1&G&\discrete\\ 
1&(1\times T, F\times 1)[4] &1+0&(2.1)&G&\gqone\\
 1&(1\times T, F^-) [2]&0+1&(2.1)&G&\gqwf\\
 1&(1\times T, F\times C_2) [4]&0+1&(2.1)&G&\gqwf\\
1&(T\times 1, 1\times F) [4]&1+0&(2.1)&G&\gqone\\
 1&(T\times 1, F^-) [2]&0+1&(2.1)&G&\gqwf\\
 1&(T\times 1, C_1\times F) [4]&0+1&(2.1)&G&\gqwf\\
\hline 
0&(1,F)&0&\mbox{Various}&&\mathrm{Discrete}\\
\hline 
\end{array}$$

{\bf Summary:}  For each rank of dominant subgroup we record the statistics of the 
number of blocks of each dimension. For those of maximal rank we 
record the number of toral, mixed and flat blocks. 

\begin{description}
\item[Rank 2] $2^51^40^1$, $5=t^1m^2f^2$
\item[Rank 1] $1^60^{10}$
\end{description}

\section{Dominant subgroups for $SU(3)$}
\label{sec:domA2}
We complete the discussion of $G=SU(3)$ by considering 
singular subgroups of rank 1. The section ends with a complete summary 
of all blocks.

\subsection{Singular subgroups for $SU(3)$}
We conisider  the case that $H_e$ is of rank 1, and  $TH$ is 
singular. 
For convenience we choose the maximal tori so that $TH\subseteq  TG$. 

The identity component of a singular subgroup of rank 1 is either a 
torus or locally $SU(2)$.  We will apply Lemma \ref{lem:gennormalizer}. 

In the latter case we are looking for a complex representation of $H$
of dimension 3; since the composite $SU(2)\lra U(3)\lra U(1)$ is 
necessarily trivial, this automatically gives a subgroup of $SU(3)$. 
The only two possibilities are $V_1\oplus V_2$ and $V_3$. 

\begin{itemize}
\item ($ V_1\oplus V_2$) This  is $1\times SU(2)$, with 
normalizer $U(2)$ and Weyl group $U(1)$. This therefore gives no new 
dominant subgroups. 

\item ($V_3$) 
This is self-normalizing. 
\end{itemize}

If the identity component is a singular circle, it has normalizer
$U(2)$ and Weyl group $PU(2)=SO(3)$. The new subgroups therefore
correspond to $H_d\in \{ A_5, \Sigma_4, A_4, D_4\}$.  We note that the
actual dominant subgroup corresponding to $H_d$ will be $Z\times_2
\widetilde{H_d}$.

\subsection{$G=SU(3)$ summary table}

Altogether, we obtain the table (we have added the omitted 
finite groups using Blichfeldt's classification \cite{BDM}). 

$$\begin{array}{cl|cccc|}
\mathrm{rk}&(H_e, F)&\dim (\cV^G_{(H_e,F)})&W_G(H)&\hat{H}& \\
\hline 
2&(G,1)&0&1&SU(3)&\mathrm{Discrete}\\
2&(U(2),1)&1+0&1&U(2)&\gqone\\
2&(T^2, D_6) &0+1&1&SU(3)&\gqwf\\
2&(T^2, C_3) &0+1&D_2&SU(3)&\gqwf\\
2&(T^2, D_2) &1+1&1&U(2) &\mix \\
2&(T^2, C_1) &2+0&D_6&T^2&\gqtoral\\
\hline 
1&(SO(3),1)&0&1&SO(3)&\mathrm{Discrete}\\
1&(Z(U(2)), A_5)&1+0&1&U(2)&\gqone\\
1&(Z(U(2)), \Sigma_4)&1+0&1&U(2)&\gqone\\
1&(Z(U(2)), A_4)&1+0&D_2&U(2)&\gqone\\
1&(Z(U(2)), D_4)&1+0&D_6&U(2)&\gqone \\
\hline 
0&(1,PSL_2(7))&0&C_3&SU(3)&\mathrm{Discrete}\\
0&(1,PSL_2(7)\times C_3)&0&1&SU(3)&\mathrm{Discrete}\\
0&(1,A_6\cdot 3)&0&1&SU(3)&\mathrm{Discrete}\\
0&(1,A_6)&0&C_3&SU(3)&\mathrm{Discrete}\\
0&(1,G_{36}\cdot 3)&0&1&SU(3)&\mathrm{Discrete}\\
0&(1,G_{72}\cdot 3)&0&1&SU(3)&\mathrm{Discrete}\\
0&(1,G_{216}\cdot 3)&0&1&SU(3)&\mathrm{Discrete}\\
\hline 
\end{array}$$

{\bf Summary:}  For each rank of dominant subgroup we record the statistics of the 
number of blocks of each dimension. For those of maximal rank we 
record the number of toral, mixed and flat blocks. 

\begin{description}
\item[Rank 2] $2^21^30^1$, $2=t^1m^2f^0$
\item[Rank 1] $1^40^{1}$
\item[Rank 0] $0^{7}$
\end{description}

\section{Dominant subgroups for $Sp(2)$}
\label{sec:domC2}
We complete the discussion of $G=Sp(2)$ by considering 
singular subgroups of rank 1. The section ends with a complete summary 
of all blocks.

For the rank 1 singular subgroups, there  are two cases: perpendicular to
a short root or the one perpendicular to a long root.  In either case,
the identity component of a singular subgroup of rank 1 is either a 
torus or locally $SU(2)$, and we will apply Lemma \ref{lem:gennormalizer}.

\subsection{Singular  $Sp(1)$ subgroups}
In the case when $H$ is locally $SU(2)$,  we are looking for a representation of $H$ over
the quaternions of dimension 1 or 2. The only symplectic simples are $V_2$ and
$V_4$, so there are three possibilities, corresponding to 
$V_1\oplus V_1\oplus V_2$, $V_2\oplus V_2$ or $V_4$. 

\begin{itemize}
\item ($V_1\oplus V_1\oplus V_2$) This  is $1\times Sp(1)$, with 
normalizer $Sp(1)\times Sp(1)$ and Weyl group $Sp(1)$ by the action
method. It's maximal
finite subgroups are $A_4,\Sigma_4, A_5, D_4$. This gives four possibilities. 

\item ($V_2 \oplus V_2$) This  is the diagonal $Sp(1)$ in $Sp(1)\times
  Sp(1)$.

    \begin{cor}
The Weyl group of $\Delta Sp(1)$ is of order 2
      \end{cor}

      \begin{proof}
By Lemma \ref{lem:gennormalizer}, the normalizer lies in $Z(T(H))\cdot H$. We may
calculate that $Z(T(H))=O(2)$. 
        \end{proof}

\item ($V_4$) This is the principal $Sp(1)$. By the representation
  method, this is self-normalizing, giving just one conjugacy class. 
\end{itemize}

\subsection{Singular circles}
We turn to the case with identity component a singular
circle. There are two cases. according to whether the circle
is perpendicular to  a long root or a short root.

\begin{itemize}
\item First we suppose $T$ is the kernel of a long root. The
  centralizer is the one corresponding to $V$ above, with
  identity component $Sp(1)\times T$ and the
  normalizer  is $Sp(1)\times NT\cong Sp(1)\times Pin (2)$, giving Weyl group $Sp(1)\times 
  C_2$. 
As in Lemma \ref{lem:counting} there are 10 possibilities.

\item Next suppose $T$ is the kernel of a short root. The centralizer
  is the subgroup corresponding to $V'$ above, which again gives
  Weyl group $Sp(1)\times 
  C_2$.  As in Lemma \ref{lem:counting} there are 10 possibilities. 
\end{itemize}

\subsection{$G=Sp(2)$ summary table}

$$\begin{array}{cl|cccc|}
\mathrm{rk}&(H_e, H_d)&\dim (\cV^G_{(H_e,F)})&W_G(H)&\hat{H}& \\
\hline 
2&(G,1)&0&1&G&\mathrm{Discrete}\\
2&(Sp(1)\times Sp(1),C_2), V&0&1&G&\mathrm{Discrete}\\
2&(Sp(1)\times Sp(1),1), V&0&C_2&Sp(1)\times Sp(1)&\mathrm{Discrete}\\
2&(Sp(1)\times Sp(1),C_2), V'&0&1&G&\mathrm{Discrete}\\
2&(Sp(1)\times Sp(1),1), V'&0&C_2& Sp(1)\times Sp(1) &\mathrm{Discrete}\\
2&(T\times Sp(1),C_2), X&0+1&1&Sp(1)\times Sp(1)&\gqwf\\
2&(T\times Sp(1),1), X&1+0&C_2&T\times Sp(1)&\gqone\\
2&(Sp(1)\times T,C_2), X'&0+1&1&Sp(1)\times Sp(1)&\gqwf\\
2&(Sp(1)\times T,1), X'&1+0&C_2&Sp(1)\times T&\gqone\\
2&(T^2, D_8) &0+1&1&\toral&\gqwf\\
2&(T^2, V) &0+2&C_2&\toral&\gqwf\\
2&(T^2, C_4) &0+1&C_2& \toral&\gqwf\\
2&(T^2, V') &0+2&C_2&\toral&\gqwf\\
2&(T^2, X) &1+1&C_2&\toral&\mix\\
2&(T^2, C_2) &0+2&D_4& \toral&\gqwf\\
2&(T^2, X') &1+1&C_2&\toral&\mix\\
2&(T^2, C_1) &2+0&D_8&\mathrm{torus}&\gqtoral\\
\hline 
1&(Sp(1)\times 1,\At_5)&0&1&Sp(1)\times Sp(1)&\mathrm{Discrete}\\
1&(Sp(1)\times 1,\Sigmat_4)&0&1&Sp(1)\times Sp(1)&\mathrm{Discrete}\\
1&(Sp(1)\times 1,\At_4)&0&C_2&Sp(1)\times Sp(1)&\mathrm{Discrete}\\
1&(Sp(1)\times 1,\Dt_4)&0&D_6&Sp(1)\times Sp(1)&\mathrm{Discrete}\\
1&(\Delta Sp(1), C_2)&0&1&G&\mathrm{Discrete}\\
1&(\Delta Sp(1), 1)&0&C_2&Sp(1)&\mathrm{Discrete}\\
1&(V_4 Sp(1), 1)&0&1&Sp(1)&\mathrm{Discrete}\\
1&(T_{long}, 1\times F)[4]&1+0&(2.1)&\toral&\gqwf\\
1&(T_{long}, F^-)[2]&0+1&(2.1)&\toral&\gqwf\\
1&(T_{long}, C_2\times F)[4]&0+1&(2.1)&\toral&\gqwf\\
1&(T_{short}, 1\times F)[4]&1+0&(2.1)&\toral&\gqone\\
1&(T_{short}, F^-)[2]&0+1&(2.1)&\toral&\gqwf\\
1&(T_{short}, C_2\times F)[4]&0+1&(2.1)&\toral&\gqwf\\
\hline 
0&(1,\mbox{Various})&0&\mbox{Various}&&\mathrm{Discrete}\\
\hline
\end{array}$$

{\bf Summary:}  For each rank of dominant subgroup we record the statistics of the 
number of blocks of each dimension. For those of maximal rank we 
record the number of toral, mixed and flat blocks. 

\begin{description}
\item[Rank 2] $2^61^60^5$, $6=t^1m^2f^3$
\item[Rank 1] $1^60^{7}$
\end{description}

\section{Models}
\label{sec:models}
Having described the partition of $\fX_{G}$ into  blocks for
$G=SU(2)\times SU(2), SU(3)$ and $Sp(2)$, we may
describe the models $\cA (G|\cV^{G}_H)$, and explain where to
find the proofs that each does give a model. The further step of
making the model explicit to the extent of making calculations is best
left to the individual groups. The model for $SU(3)$ was described previously
in \cite{su3q}, and the detail of the structure for $Sp(2)$ will be
given in \cite{sp2q2}.

\subsection{Dimension 0} 
There are a large number  0-dimensional singleton blocks. By definition their
Weyl group is finite, so the data for the
one dominated by $H$ is just the finite group $W_G(H)$.
 It is shown in
\cite{gfreeq2, spcgq} that $\Q[W_G(H)]$-modules give a model.

\subsection{Dimension 1} 
There are many 1-dimensional blocks. The data for these consists
of a sheaf of rings and a component structure.

The general structure of a 1-dimensional block can be combinatorially
complicated, but all 1-dimensional blocks considered here are either
cotoral lines or flat lines (i.e.,  like one of the two blocks occurring for $O(2)$).

Each flat line is the one point compactification of a countable 
discrete set of minimal elements. The single compactifying subgroup is 
the dominant subgroup. In most cases, the minimal elements are 
finite subgroups, but there are four in which the  identity component is 
$Sp(1)$. Flat lines are Weyl-finite, so the additional structure 
is the component structure, given by the Weyl groups associated to 
each point. The model consists of the equivariant 
sheaves over the space. This is proved to be a model in \cite{gqwf}.

Each cotoral line is formed from a countable set of minimal points as
$\spec (\Z)$ is formed from the non-zero primes.  The generic point
 is the dominant subgroup. In most cases, the minimal elements are 
finite subgroups, but there are five in which the  identity component
is of dimension 3. The sheaf of rings is a polynomial ring on one variable over
each minimal point, and $\Q$ over the generic point. 
There is also the component structure, given by the Weyl component 
groups at each point. The model is given by equivariant versions of $\cA (SO(2))$.
These are proved to be  models in \cite{gq1}.

\subsection{Dimension 2} 
For $G=SU(3)$ a full account was given in \cite{su3q}, but we see
there are just two 2-dimensional blocks. One is toral (a special case 
of \cite{gtoralq}), and one is mixed (as in \cite{t2wqmixed}). 

For $G=SU(2)\times SU(2)$, there are five 2-dimensional blocks. 
Two are  Weyl-finite (dealt
with by \cite{gqwf}), one is toral (a special case
of \cite{gtoralq}), and two are mixed  (covered in \cite{t2wqmixed}).

For $G=Sp(2)$, there are six 2-dimensional blocks.
One is toral (a special case 
of \cite{gtoralq}). At the other extreme, three are Weyl-finite blocks as in
\cite{gqwf}.  To make these explicit, we need to identify the space of conjugacy
classes. Two are products of  flat lines, and one is a
2-dimensional Grassmannian.  Each of these has  a component structure,
 and the models consist of equivariant sheaves. 

Finally,  two are mixed blocks, once again of the form discussed in
 \cite{t2wqmixed}: up to a finite ramified covering they are the
 product of a cotoral line and a flat line.

\end{document}